\documentclass[12pt, reqno]{amsart}
\usepackage{amsmath, amsthm, amscd, amssymb, graphicx, color}
\usepackage{lmodern}
\usepackage{mathrsfs}

\usepackage{hyperref}
\usepackage{enumitem}
\newtheorem{theorem}{Theorem}[section]
\newtheorem{lemma}[theorem]{Lemma}
\newtheorem{prop}[theorem]{Proposition}
\newtheorem{corollary}[theorem]{Corollary}
\theoremstyle{definition}
\newtheorem{definition}[theorem]{Definition}

\theoremstyle{remark}

\numberwithin{equation}{section}

\newcommand{\h}{\mathcal{H}}
\newcommand{\m}{\mathcal{M}}
\newcommand{\bn}{\overline{B_n}}
\newcommand{\f}{\mathcal{F}}
\begin{document}
	
	\title[Near Invariance of The Dual Compressed Shift]{Near Invariance of The Dual Compressed Shift}
	
	\author[Arup Chattopadhyay and Supratim Jana]{Arup Chattopadhyay$^*$ and Supratim Jana$^{**}$}
	\maketitle
	
	\paragraph{\textbf{Abstract}}
      We present the notion of nearly dual compressed shift-invariant subspaces of the orthogonal complement of the model space in $L^2$ and obtain their structure using the celebrated Hitt-Sarason algorithm \cite{DH, SAR2}.
	
	\vspace{0.5cm}
    
    
	\vspace{0.5cm}
    
	
	\section{\textbf{Introduction and Preliminaries}}
	
	We denote the open unit disk in the complex plane by $ \mathbb{D}:=\{ z \in \mathbb{C}: |z| < 1 \} $ and its boundary by $ \mathbb{T}:=\{ z \in \mathbb{C}: |z|=1 \} $. The \emph{Hardy-Hilbert space} $H^2(\mathbb{D})$ is defined as the space of analytic functions on $\mathbb{D}$ given by
	\[
	H^2(\mathbb{D}) = \left\{ f(z) = \sum_{n=0}^\infty a_n z^n \; : \; \sum_{n=0}^\infty |a_n|^2 < \infty \right\}.
	\]
    
	Also, let $L^2$ denote the space of square integrable functions on $\mathbb{T}$ with respect to the normalized Lebesgue measure and $L^\infty$ denote the von Neumann algebra of essentially bounded Lebesgue measurable functions on $\mathbb{T}.$
    
    By the elegance of Fatou's theorem (\cite{FM1, GMR, ARO}), the space $H^2(\mathbb{D})$ is identified as a closed subspace ($H^2(\mathbb{T})$) of $L^2,$ whose every element has only non-negative Fourier coefficients. For convenience, we write $H^2$ for the Hardy–Hilbert space irrespective of $\mathbb{D}$ and $\mathbb{T}$, according to the context.
	
	An \emph{inner function} $\theta$ is a bounded analytic function on $\mathbb{D}$ whose boundary values have unit modulus, that is, $|\theta(z)|=1$ almost everywhere on $\mathbb{T}$. The shift operator on $H^2$, denoted by $S$, is defined by
	$S(f) = zf.$
	The celebrated Beurling's theorem \cite{AB}, a cornerstone in function theory, characterizes all nontrivial shift-invariant subspaces of $H^2$, showing that they are precisely of the form $\theta H^2$, where $\theta$ is an inner function. As a consequence, the \emph{model space}
	$K_\theta = H^2 \ominus \theta H^2$
	arises as the invariant subspace of the adjoint of $S$, denoted by $S^*$, and is defined as
	\[
	S^*(f)(z) = \frac{f(z)-f(0)}{z} = P(\bar{z}f),
	\]
	where $P: L^2 \to H^2$ denotes the orthogonal projection(also known as the Szeg\"o projection).

    In the mid of 1960's, Brown-Halmos introduced the Toeplitz operator in \cite{BH}, and defined this as
    \[ T_\phi:H^2\rightarrow H^2 ~\text{ by }~ T_\phi(f) = P(\phi f). \] Therefore, $T_\phi$ can be expressed as the compression of the multiplication operator $M_\phi$ on $L^2$ to the Hardy space $H^2$.
	
	In 2007, Sarason, in his seminal paper \cite{SAR1}, introduced a further compression of the Toeplitz operator on the model space $K_\theta$, which he termed the \emph{truncated Toeplitz operator} (TTO). The truncated Toeplitz operator, denoted by $A_\phi^\theta$, is a densely defined operator given by
	$
	A_\phi^\theta(h) = P_\theta(\phi h),
	$
	where $\phi \in L^2(\mathbb{T})$, $h \in K_\theta$ with $\phi h \in L^2(\mathbb{T})$, and $P_\theta : L^2(\mathbb{T}) \to K_\theta$ denotes the orthogonal projection. In recent years, this topic has attracted considerable attention from numerous mathematicians.
	
	Recently, Ding-Sang came up with the study of \emph{dual truncated Toeplitz operators} (DTTO) in their article \cite{DS}. This operator is defined on the orthogonal complement of the model space $K_\theta^\perp$ by
	\[
	D_\phi(h) = (I - P_\theta)(\phi h)=Q_\theta(\phi h), \quad h \in K_\theta^\perp,
	\]
	where $I-P_\theta=Q_\theta : L^2(\mathbb{T}) \to K_\theta^\perp$ denotes the orthogonal projection and $\phi \in L^\infty$. Subsequently, Camara et. al. in \cite{CKLP} and Gu in \cite{CG1} provided further characterizations of DTTO and generalized several results established in \cite{DS}.
    \vspace{0.1in}
	
	In this regard, we recall the following two operators. Firstly, the \emph{compressed shift} operator $S_\theta$, which is actually the compression of the (unilateral) shift operator $S$ on the model space $K_\theta$, is defined by 
    \[S_\theta(f) = P_\theta(zf), ~ f \in K_\theta.\]
    And its dual operator,
	\[
	D_z(h)=U_\theta(h) =  Q_\theta(zh), ~ h \in K_\theta^\perp,
	\] 
    is called the \emph{dual of the compressed shift} $S_\theta$. This operator may be viewed as the compression over $K_\theta^\perp$ of the bilateral shift on $L^2$. For convenience, we use $U_\theta$ to denote the dual-compressed shift throughout the article. 
    \vspace{0.1in}
    
    The recent works of Camara-Ross \cite{CR} and  Li et al.~\cite{LSD} initiated the study of $U_\theta$-invariant subspaces of $K_\theta^\perp$. Later, a substantial amount of development in the theory of invariant subspaces related to $U_\theta$ has been performed by Gu, Luo, Timotin \cite{CG2, GL, DT} and others.
	
	In 1988, while studying invariant subspaces of the Hardy space on the annulus, ~Hitt introduced the concept of \emph{nearly $S^*$-invariant} subspaces of $H^2,$ and provided a complete structural characterization. Subsequently, Sarason refined Hitt's argument in \cite{SAR2}. First, we recall the definition.
	
	\begin{definition}
		A closed subspace $M$ of $H^2$ is called \emph{nearly $S^*$-invariant} subspace if for every $f\in M$ with $f(0)=0 $ implies $ S^*f \in M$.
	\end{definition}
	
	We now reformulate the above definition in several equivalent ways, as follows. A closed subspace $M \subset H^2$ is called \emph{nearly $S^*$-invariant} if one of the following conditions holds:
	\begin{align}\label{3eqd1}
	(i) \quad & S^*(M \cap 1^\perp) \subset M, \\
	(ii) \label{3eqd2} \quad & S^*(M \cap zH^2) \subset M, \\
	(iii) \label{3eqd3} \quad & S^*(M \cap SH^2) \subset M.
	\end{align}
	
	Since the shift operator $S$ is left-invertible (as $S^*S = I$ on $H^2$), Liang and Partington \cite{LP} introduced the concept of a nearly $T^{-1}$-invariant subspace of a Hilbert space $\mathcal{H}$, defined by the condition
	$T^{-1}(M \cap T\mathcal{H}) \subset M,$ where $T^{-1}$ denotes the left inverse of a bounded operator $T$ on $\mathcal{H}$. Previously, Camara and Partington \cite{CP} introduced and studied the notion of nearly $\eta$-invariant subspaces of $H^2$. In recent years, several related forms of near-invariance have been defined and investigated (see, for example, \cite{ABBH, DSa, KLS}).
	
	In analogy with condition~$(iii)$ of the preceding definition, we now introduce the notion of a nearly invariant subspace for a general bounded operator on $\mathcal{H}$ as follows.
	
	\begin{definition} \label{3D}
		Let $\mathcal{H}$ be a Hilbert space, and let $T: \mathcal{H} \rightarrow \mathcal{H}$ be a bounded operator. A closed subspace $\mathcal{F}$ of $\mathcal{H}$ is called \emph{nearly $T$-invariant} subspace of $\mathcal{H}$ if $ T( \mathcal{F} \cap T^* \h ) \subset \mathcal{F}.$ 
	\end{definition}

    Note that, with this definition, we can also characterize the nearly $S$-invariant subspaces of $H^2$. They would be the Beurling subspaces $\theta H^2$ (since, $S( \mathcal{F} \cap S^* H^2 ) \subset \mathcal{F} \implies S( \mathcal{F} ) \subset \mathcal{F}$).
	
	As a consequence, we define a \emph{nearly dual compressed shift-invariant} subspace (or equivalently \emph{nearly $U_\theta$-invariant} subspace) of $K_\theta^\perp$ as follows.
	
	\begin{definition}
		A closed subspace $\mathcal{M} \subset K_\theta^\perp$ is called a \emph{nearly dual compressed shift-invariant} subspace (or \emph{nearly $U_\theta$-invariant}) if $ U_\theta (\mathcal{M}\cap U_\theta^* K_\theta^\perp) \subset \mathcal{M}$.
	\end{definition}
	
	Similarly, a closed subspace $\mathcal{M} \subset K_\theta^\perp$ is called \emph{nearly $U_\theta^*$-invariant} if  $U_\theta^*(\mathcal{M}\cap U_\theta K_\theta^\perp) \subset \mathcal{M}$.
	\vspace{0.1in}
	
	It is essential to observe that the behavior of $U_\theta$ and $U_\theta^*$ is determined by whether $\theta(0)=0$ or $\theta(0)\neq 0$. Specifically, none of them is an isometry when $\theta(0)=0$.
	\vspace{0.1in}
	
	The main objective of this article is to characterize the structure of nearly $U_\theta$ (and $U_\theta^*$ -) invariant subspaces $\mathcal{M} \subset K_\theta^\perp$, irrespective of the value of $\theta(0)$, that is of the form $\mathcal{M} = \mathcal{M}_- \oplus \mathcal{M}_+,$ where $\mathcal{M}_-$ is a closed subspace of $\overline{H_0^2} \,(= \overline{zH^2})$, and $\mathcal{M}_+$ is a closed subspace of $\theta H^2$. And it is the best we can obtain. Indeed, Câmara-Ross \cite{CR}, as well as Li et al.~\cite{LSD}, have constructed examples of $U_\theta$-invariant subspaces which cannot be expressed in the form $\mathcal{M}_- \oplus \mathcal{M}_+$, and hence the same limitation applies to nearly $U_\theta$-invariant subspaces. 
	
	The article is organized as follows. In Section 2, we obtain the structure of a nearly $U_\theta^*$-invariant subspace, when $\theta(0)\neq 0 $. In Section 3, we obtain the same for the case $\theta(0)=0$ (see Theorem~\ref{3mainthm}) by using the celebrated Hitt-Sarason algorithm. Moreover, as a consequence of our main results, we also discussed the splitting subspaces of $H^2_{\mathbb{C}^2}$ that are nearly $S\oplus S^*$-invariant (see Corollary~\ref{3mainconsequence}).
	
	\section{Preparatory Facts}
	
	In this section, we recall some important results on the dual compressed shift, established by several mathematicians. To begin with, we present the explicit action of the dual compressed shift $U_\theta$, its adjoint $U_\theta^*$, and their products on $K_\theta^\perp$, as obtained by C.~Gu in \cite{CG1}. 
	
	\begin{lemma}[Lemma 2.1 and 2.2 \cite{CG1}]\label{3Glemma}
		For $h\in K_\theta^\perp,$ we have 
		\begin{align*}
		& (i) U_\theta (h)= zh + \langle h, \bar{z} \rangle ( \overline{\theta_0} \theta - 1 ),\\
		& (ii) U_\theta^*(h)= \bar{z}h + \langle h, \theta \rangle ( \theta_0 - \theta) \bar{z}, \\
		& (iii) I - U_\theta^* U_\theta = (1-|\theta_0|^2) \bar{z} \otimes \bar{z},\\
		& (iv) I- U_\theta U_\theta^* = (1-|\theta_0|^2) \theta \otimes \theta,
		\end{align*}
		where $\theta_0$ stands for both $\theta(0)$ (when $\theta$ is considered to be analytic on $\mathbb{D}$) and the $0^{th}$ Fourier coefficient (when $\theta$ is considered as a boundary function) of the inner function $\theta.$ Note that, for a Hilbert space $\mathcal{H}$, and for $f, g\in \mathcal{H}$, $f\otimes g$ denotes the rank one operator on $\mathcal{H}$ defined by $f\otimes g(h)=\langle h, g\rangle f$. 
	\end{lemma}
	
	Next, we present the invariant subspaces of $U_\theta$, as obtained by Camara-Ross \cite{CR}, and independently by Li et al.~\cite{LSD}.

	\begin{theorem}\label{3Th}
		Let $\theta$ be an inner function, and let $\mathcal{F}$ be a non-trivial $U_{\theta}$ invariant subspace of $K_\theta^{\perp}$ of the form $\mathcal{F} = X_- \oplus Y_+$, where $ X_- $ is closed subspace of $\overline{H_0^2}$ and $Y_+$ is a closed subspace of $\theta H^2$.
		\vspace{0.1in}

		(i) If $\theta_0 \neq 0,$ then $\mathcal{F}$ takes one of the forms: $\gamma \theta H^2$ or $ \overline{ zK_\alpha } \oplus \theta H^2, $ where $\gamma$ and $\alpha$ are inner functions.
		\vspace{0.1in}
		
		(ii) If $\theta_0 = 0,$ then $\mathcal{F}$ takes one of the forms: $ \overline{H_0^2},~ \overline{zK_\alpha},~ \gamma \theta H^2,~ \overline{H_0^2 } \oplus \gamma \theta H^2 $ or $\overline{zK_\alpha} \oplus \gamma \theta H^2$, where $\alpha$ and $\gamma$ are inner functions.
	\end{theorem}
	
	As an application of the above theorem, we obtain the structure of nearly $U_\theta$ invariant subspaces, when $\theta_0 \neq 0.$
	
	\begin{theorem}
		Let $\theta$ be an inner function with $\theta_0 \neq 0$. If $\mathcal{M}$ is a non-trivial nearly $U_\theta$ invariant subspace of $K_\theta^{\perp}$ of the form $\mathcal{M} = \mathcal{M_-} \oplus \mathcal{M_+}$, where $ \mathcal{M_-} $ is closed subspace of $\overline{H_0^2}$ and $\mathcal{M_+}$ is a closed subspace of $\theta H^2$, then $\mathcal{M}$ is either $\gamma \theta H^2$ or $ \overline{ zK_\alpha } \oplus \theta H^2, $ where $\gamma$ and $\alpha$ are inner functions.
	\end{theorem}
	
	\begin{proof}
		First note that $U_\theta^* ( K_\theta ^\perp )=K_\theta ^\perp= U_\theta ( K_\theta ^\perp ) $. Therefore, by Definition 1.2, we have 
		$$ U_\theta ( \mathcal{M} \cap U_\theta^* K_\theta ^\perp ) \subset \mathcal{M} \implies U_\theta ( \mathcal{M} \cap K_\theta ^\perp ) \subset \mathcal{M} \implies U_\theta \mathcal{M} \subset \mathcal{M}.$$ Hence, by the Theorem \ref{3Th}, we get $ \mathcal{M} = \gamma \theta H^2$ or $ \overline{ zK_\alpha } \oplus \theta H^2 ,$ where $\gamma$ and $\alpha$ are inner functions.
	\end{proof}
	
	Now, if we replace $U_{\theta}$ by $U_{\theta}^*$ in the above theorem, we have the following.
	\begin{corollary}
		Let $\theta$ be an inner function with $\theta_0 \neq 0$. If $\mathcal{M}$ is a non-trivial nearly $U_\theta^*$-invariant subspace of $K_\theta^{\perp}$ of the form $\mathcal{M} = \mathcal{M_-} \oplus \mathcal{M_+}$, where $ \mathcal{M_-} $ is closed subspace of $\overline{H_0^2}$ and $\mathcal{M_+}$ is a closed subspace of $\theta H^2$, then $\mathcal{M}$ is either $\theta K_\gamma \oplus \overline{H_0^2}$ or $ \overline{\alpha H_0^2}, $ where $\gamma$ and $\alpha$ are inner functions.
	\end{corollary}
	We next define an operator on $\overline{H_0^2}$ by  
	$\mathbb{S}: \overline{H_0^2} \to \overline{H_0^2}, 
	\quad \mathbb{S}(\bar{h}) = \bar{z}\,\bar{h}, \,\bar{h} \in \overline{H_0^2}.$
	Its adjoint is given by  
	$
	\mathbb{S}^* : \overline{H_0^2} \to \overline{H_0^2}, 
	\quad \mathbb{S}^*(\bar{h}) = (I-P)(z \bar{h}),
	$
	where $P$ denotes the orthogonal projection onto $H^2$. 
	It is straightforward to verify that the operators $\mathbb{S}$ and $\mathbb{S}^*$ are unitarily equivalent to the classical shift $S$ and the backward shift $S^*$, respectively, acting on $H^2$. 
	Consequently, the nontrivial invariant subspaces of $\mathbb{S}$ and $\mathbb{S}^*$ in $\overline{H_0^2}$ are of the form 
	$
	\overline{\psi H_0^2} 
	\, \, \text{and} \,\,
	\overline{z K_\xi},
	$
	respectively, where $\psi$ and $\xi$ are inner functions. 
	
	Similarly, we define the operator $\mathscr{S} : \overline{H^2} \to \overline{H^2}, 
	\, \mathscr{S}(\bar{h}) = \bar{z}\,\bar{h}, 
	\, \bar{h} \in \overline{H^2}.$
	Its adjoint is given by  
	$
	\mathscr{S}^*(\bar{h}) = z\bigl(\bar{h} - \bar{h}_0 \bigr),
	$
	where $\bar{h}_0$ denotes the constant term of $\bar{h}$. 
	It follows that the invariant subspaces of $\mathscr{S}$ are precisely of the form $\overline{\alpha H^2},$
	while the invariant subspaces of $\mathscr{S}^*$ are of the form  
	$
	\overline{K_\alpha},
	$
	for some inner function $\alpha$.
	
	In what follows, we present another characterization of the dual compressed shift operator $U_\theta$.
	
	\begin{prop}\label{3TH}
		Let $T$ be an operator acting on $K_\theta^\perp$ as $$ T(h)= \begin{cases} Sh , & h \in \theta H^2 \\ \mathbb{S}^*h , & h \in \overline{H_0^2}.\end{cases}$$ When $\theta_0 = 0,$ then $U_\theta = T.$
	\end{prop}
	
	\begin{proof}
		Let $h \in K_\theta ^\perp$ be such that $h= \theta h_+ + h_-$, where $h_+ \in H^2$ and $h_- \in \overline{H_0^2}$. Now if $\theta_0=0,$ then by using Lemma~\ref{3Glemma} we have
		\begin{align*}
		& U_\theta h = zh +\langle h, \bar{z} \rangle ( 0-1 ) \\
		& = zh- \langle zh, 1 \rangle \\
		& = z\theta h_+ + zh_- - \langle z\theta h_+ + zh_-, 1 \rangle,\\
		& = z\theta h_+ + ( zh_- - \langle h_-, \bar{z} \rangle ) = T(h), \text{ }
		\end{align*} 
		and hence proved.
	\end{proof}
	As a consequence of the above fact, we can acquire $U_\theta$ invariant subspaces of the form $\mathcal{M}= \mathcal{M}_+ \oplus \mathcal{M}_- $ for the case $\theta_0=0$. 
	\begin{corollary}
		In the case $\theta_0=0,$ the $U_\theta$ invariant subspaces are merely of the form: \big($S$ invariant subspaces of $\theta H^2$\big) $\oplus$ \big($\mathbb{S}^*$ invariant subspaces of $ \overline{H_0^2}\big).$
	\end{corollary}
	
	\section{Main Results}
	In this section, we investigate the structure of nearly $U_\theta$-invariant subspaces of the form  
	$ \mathcal{M} = \mathcal{M}_- \oplus \mathcal{M}_+,$
	in the case when $\theta_0 = 0$. Our approach proceeds in parallel with the elegant proof of Hitt’s theorem \cite{DH}, which has appeared in various contexts across the literature (see, for instance, \cite{ABBH, CCP, CGP, FM1, FM, GMR, DH, EH, SAR2}).  
	
	Now, by the definition of a nearly $U_\theta$-invariant subspace of $K_\theta^\perp$, we have $U_\theta \bigl( \mathcal{M} \cap U_\theta^* K_\theta^\perp \bigr) \subset \mathcal{M}.$
	Since $U_\theta^* K_\theta^\perp = K_\theta^\perp \ominus \bigvee \{\bar{z}\},$ it follows that $U_\theta \bigl( \mathcal{M} \cap \{\bar{z}\}^\perp \bigr) \subset \mathcal{M}.$  
	Thus, this notion is directly analogous to the definitions \eqref{3eqd1} and \eqref{3eqd3} introduced earlier.
	
	A natural question arises: is it possible to have a non-zero nearly $U_\theta$-invariant subspace $\mathcal{M}$ such that $\mathcal{M} \subset K_\theta^\perp \ominus \big(\bigvee \{\bar{z}\}\big)\, ?$
	
	The answer is affirmative, unlike the case of nearly $S^*$-invariant subspaces of the Hardy space $H^2$.  
	\begin{theorem}
		Let $\mathcal{M}$ be a non-trivial nearly $U_\theta$-invariant subspace that is contained in 
		$
		K_\theta^\perp \ominus \big(\bigvee \{\bar{z}\}\big).
		$
		Then $\mathcal{M}$ is of the form 
		$
		\mathcal{M} = \alpha \, \theta H^2,
		$
		where $\alpha$ is an inner function.
	\end{theorem}
	
	\begin{proof}
		Suppose, $\mathcal{M}$ is contained in $K_\theta^\perp \ominus  \big(\bigvee \{\bar{z}\}\big)$. Then $f\in \mathcal{M} = \mathcal{M}_- \oplus \mathcal{M}_+ \implies f= f_- \oplus \theta f_+$. Therefore, $ U_\theta^n  f = {\mathbb{S}^*}^nf_- \oplus S^n\theta f_+ \in \mathcal{M} $ for every $n\in \mathbb{N}$. This implies that $f_- = 0,$ and hence $\mathcal{M}_- = \{0\}.$ Thus, $\mathcal{M}= \mathcal{M}_+$ and in this case $ U_\theta \mathcal{M} \subset \mathcal{M} \implies S\mathcal{M_+} \subset \mathcal{M_+} \implies \mathcal{M} = \alpha \theta H^2, $ where $\alpha$ is an inner function (by Beurling's theorem\cite{AB}).
	\end{proof}
	Now, a second natural question arises: what can be said if 
	$
	\mathcal{M} \not\subset K_\theta^\perp \ominus \bigvee \{\bar{z}\}\,?
	$
	\begin{lemma} \label{3L1}
		If $\mathcal{M}$ is a nearly $U_\theta$-invariant subspace of $ K_\theta^\perp $ that is not contained in $K_\theta^\perp \ominus \big(\bigvee \{\bar{z}\}\big)$, then $ \mathcal{M} \ominus (\mathcal{M} \cap \{\bar{z}\}^\perp) $ is a one-dimensional space.
	\end{lemma}
	
	\begin{proof}
		Suppose that $g$ is an element of $  \mathcal{M} \ominus (\mathcal{M} \cap \{\bar{z}\}^\perp) ,$ with $||g||=1.$ So, $\langle g,\bar{z} \rangle \neq 0.$ Now, consider $ f$ is an arbitrary element of $  \mathcal{M} \ominus (\mathcal{M} \cap \{\bar{z}\}^\perp) $. To show that $ \mathcal{M} \ominus (\mathcal{M} \cap \{\bar{z}\}^\perp) $ is one-dimensional, we prove that $ \mathcal{M} \ominus (\mathcal{M} \cap \{\bar{z}\}^\perp) = \mathbb{C} g.$ Since, $f\in  \mathcal{M} \ominus (\mathcal{M} \cap \{\bar{z}\}^\perp) ,$ we have $\langle f, \bar{z} \rangle \neq 0.$ Now, consider the element $h= f- \frac{\langle f, \bar{z}\rangle}{\langle g, \bar{z}\rangle}g.$ Clearly $ h $ is an element of $  \mathcal{M} \ominus (\mathcal{M} \cap \{\bar{z}\}^\perp) .$ But, $\langle h, \bar{z} \rangle =0,$ which implies that $ h \in \big( \mathcal{M} \ominus (\mathcal{M} \cap \{\bar{z}\}^\perp) \big) \cap \big(\mathcal{M} \cap \{\bar{z}\}^\perp\big) = \{0\}.$ Hence, every element of $ \mathcal{M} \ominus (\mathcal{M} \cap \{\bar{z}\}^\perp) $ is just a scalar multiple of $g.$
	\end{proof}
	Note that  
	$
	\bigvee \{g\} = \mathcal{M} \ominus \bigl(\mathcal{M} \cap \{\bar{z}\}^\perp \bigr),
	$ 
	and therefore $g$ cannot lie inside $\theta H^2$ entirely. In this situation, one may further ask: is it possible that some component of $g$ belongs to $\theta H^2$, that is,  
	$
	P(g) \neq 0 \, ?
	$
	\vspace{0.1in}
	
	The following lemma provides a negation to this question.
	
	\begin{lemma} \label{3L2}
		The element $g \in  \mathcal{M} \ominus (\mathcal{M} \cap \{\bar{z}\}^\perp) $ has no non-zero positive Fourier coefficients.
	\end{lemma}
	
	\begin{proof}
		Let $\mathcal{H}$ be a Hilbert space, and let $A, B, C, D$ be closed subspaces of $\mathcal{H}$ such that $A \perp D,  ~C \perp D, ~\text{and}~ B \subset D.$ Then, $ (A\oplus B ) \cap (C \oplus D) = (A \cap C) \oplus B. $ Note that, we have $ \mathcal{M} = \mathcal{M}_- \oplus \mathcal{M}_+,$ and hence, $ \mathcal{M} \cap \{\bar{z}\}^\perp = (\mathcal{M}_- \oplus \mathcal{M}_+) \cap (\overline{z H_0^2} \oplus \theta H^2).$ Considering, $ A= \mathcal{M}_-, B = \mathcal{M}_+, C = \overline{zH_0^2} , \text { and } D = \theta H^2, $  we have $ \mathcal{M} \cap \{\bar{z}\}^\perp = ( \mathcal{M}_- \cap \overline{zH_0^2}) \oplus \mathcal{M}_+.$ 
		Therefore, 
		\begin{align*}
		& g\in ( \mathcal{M} \ominus (\mathcal{M} \cap \{\bar{z}\}^\perp) )\\
		& \implies g \in \mathcal{M} \cap \big( ( \mathcal{M}_- \cap \overline{zH_0^2}) \oplus \mathcal{M}_+ \big)^\perp\\
		& \implies g \in \mathcal{M}\cap \big(  (   \mathcal{M}_- \cap \overline{zH_0^2}  ) ^\perp \cap \mathcal{M}_+ ^ \perp  \big)\\
		& \implies g \in ( \mathcal{M} \cap \mathcal{M}_+^\perp ) \cap (  \mathcal{M}_- \cap \overline{zH_0^2} )^\perp \\
		& \implies g \in \mathcal{M}_- \cap  (\mathcal{M}_- \cap \overline{zH_0^2})^\perp = \mathcal{M}_- \ominus (\mathcal{M}_- \cap \overline{zH_0^2}) \subset \overline{H_0^2}.
		\end{align*}
		This completes the proof. 
	\end{proof}
	According to Lemma~\ref{3L1}, if $\mathcal{M}$ is a non-trivial nearly $U_\theta$-invariant subspace in the case $\theta_0 = 0$, then  
	$
	\mathcal{M} = \bigl(\mathcal{M} \cap \{\bar{z}\}^\perp \bigr) \oplus \bigl( \bigvee \{ g \} \bigr).
	$
	Moreover, Lemma~\ref{3L2} ensures that $(I-P)(g) = g$. Hence, without loss of generality, we may take $g$ such that $\|g\| = 1$, and (define)  
	$
	(g)_{-1} := \langle g, \bar{z} \rangle \in \mathbb{R}.
	$ 
	With this setup, we arrive at the following observation.
	\begin{corollary}\label{3C1}
		Let $\mathcal{M}$ be a non-trivial nearly $U_\theta$-invariant subspace of the form $\mathcal{M}= \mathcal{M}_- \oplus \mathcal{M}_+$, and let $g$ be the unique element orthogonal to $\mathcal{M} \cap \{\bar{z}\}^\perp$ with unit norm such that $(g)_{-1}$ is a real number. Then, for any $h\in \mathcal{M},$ we have $ (h)_{-1}= \langle h, (g)_{-1}g \rangle. $
	\end{corollary}
	
	\begin{proof}
		Since, $\mathcal{M}= (\mathcal{M} \cap \{\bar{z}\}^\perp) \oplus \big(\bigvee\{g\}\big),$ we have: 
		\begin{enumerate}[label=(\roman*)]
			\item if $h\in \mathcal{M} \cap \{\bar{z}\}^\perp,$ then $(h)_{-1} = \langle h, \bar{z} \rangle =0= \langle h, g \rangle = \langle h, (g)_{-1} g \rangle,$ and
			\item if $h\in \bigvee\{g\},$ then $h = cg,$ for some $c\in \mathbb{C}$. Therefore, $ (h)_{-1}= c(g)_{-1}= \langle cg, (g)_{-1} g \rangle = \langle h, (g)_{-1}g \rangle,$ and hence proved. 
		\end{enumerate}
	\end{proof}
	
	In the following, we establish a connection between a nearly $U_\theta$-invariant subspace and a subspace that remains invariant under the action of a specific rank-one perturbation of $U_\theta$.
	
	\begin{lemma}\label{3L3}
		If $\mathcal{M}$ is a nearly $U_\theta$-invariant subspace of $K_\theta ^\perp$ of the form $\mathcal{M}=\mathcal{M}_- \oplus \mathcal{M}_+,$ then $\mathcal{M}$ is an invariant subspace under the action of $R_g$ operator, where $ R_g(h)= U_\theta ( I- g\otimes g )(h), h \in K_\theta^\perp,$ and $g$ carry the same meaning as in Corollary \ref{3C1}. 
	\end{lemma}
	
	\begin{proof}
		For $h \in \mathcal{M},$ we have 
		\begin{align*}
		& R_g(h) = U_\theta ( I- g\otimes g )(h)\\
		& = U_\theta ( h- \langle h, g\rangle g )\\
		& = U_\theta \left( h- \frac{(h)_{-1}}{(g)_{-1}}g \right)= U_\theta ( \psi )( \text{ say} ).
		\end{align*}
		Observe that the element $\psi \in \mathcal{M} \cap \{\bar{z}\}^\perp,$ and hence by the definition of nearly $U_\theta$ -invariance, $U_\theta(\psi)\in \mathcal{M}.$ Thus, $\mathcal{M}$ is invariant under $R_g$ operator. 
	\end{proof}
	
	Now, we are set to apply the Hitt-Sarason algorithm, which in turn yields the following principal result of this article.

        \begin{theorem}\label{3mainthm}
        Let $\mathcal{M}$ be a non-trivial nearly dual compressed shift $U_\theta$-invariant subspace of the form $\mathcal{M}=\mathcal{M}_-\oplus\mathcal{M}_+.$ Then $\mathcal{M}$ is of one of the following forms: 
        $$ \theta H^2,\quad \gamma \theta H^2, \quad g \overline{K_\alpha}, \quad g \overline{K_\alpha} \oplus \gamma \theta H^2,\quad g \overline{K_\alpha} \oplus \theta H^2,$$ 
        $$ g_1\overline{H_0^2},\quad g_1\overline{H_0^2} \oplus \gamma\theta H^2, \text{ and } g_1\overline{H_0^2}\oplus \theta H^2,$$
        where $g\in\overline{H_0^2}$ is as described in Corollary~\ref{3C1}, and $\alpha$ and $\gamma$ are inner functions. And in the last three cases, $g=g_1\bar{z},$ where $\overline{g_1}$ is an inner function satisfying $\overline{g_1}(0)\neq 0$.
        \end{theorem}
	
	\begin{proof}
		Let $\theta_0=0$ and let $f\in \mathcal{M}.$ Then $f$ can be written as $$ f= a_0 g + f_1, $$ where $a_0g = P_{ \mathcal{M}\ominus ( \mathcal{M} \cap \{\bar{z}\}^\perp ) } ( f )= P_{\mathbb{C}g} (f) $ and $f_1 \in ( \mathcal{M} \cap \{\bar{z}\}^\perp) $. Then, by the definition of near invariance for $U_\theta,$ we have $ U_\theta f_1= F_1, $ for some $F_1 \in \mathcal{M}.$
		
		So, we have: $ U_\theta^* F_1 = U_\theta^* U_\theta f_1= ( I-\bar{z}\otimes \bar{z} ) f_1= f_1.$ Also, by applying $U_\theta$ further  on both the sides, we have $ U_\theta f_1= F_1 \implies  U_\theta U_\theta^* F_1 = F_1  \implies F_1 \perp \theta.$ Therefore, $U_\theta^* F_1= \bar{z}F_1.$
		On the other hand, \begin{align*}
		&  f_1= U_\theta^* F_1 = f- a_0 g = ( Id- g \otimes g ) ( f )\\
		& \implies F_1 = U_\theta ( Id- g \otimes g ) f = R_g(f) \text{ (say).}
		\end{align*}
		Moreover, 
		\begin{align*}
		& R_g(f)= U_\theta (f - \langle f, g \rangle g)\\
		& = T(f) - \langle f_- + f_+ , g \rangle \mathbb{S}^*g\\
		& = \mathbb{S}^* (f_-) + \langle f_-, g \rangle \mathbb{S}^*g + S(f_+)  \\
		& = r_g(f_-) + zf_+,
		\end{align*}
		where
		\begin{align*}
		&   f= f_- + f_+, \quad f_-\in \mathcal{M}_-,~ f_+ \in \mathcal{M}_+, ~\text{and}~ r_g(f_-) = \mathbb{S}^*( I - g\otimes g ) (f_-).
		\end{align*}	
		Therefore, $$ f= a_0g + U_\theta^* F_1= a_0g+ U_\theta^* ( r_g(f_-)+zf_+ )  $$
		$$ \hspace{1in} = a_0g+ \mathbb{S}(r_g(f_-)) + f_+, $$
		$$ \text{ and } ||f||^2= |a_0|^2 +||f_+||^2 + ||r_g(f_-)||^2. $$
		Now, further expanding $f$, we get $$ f=a_0g + U_\theta^*( a_1g+f_2 ), \text{ where $f_2 \in \mathcal{M}\cap \{\bar{z}\}^\perp$ }. $$ Again, by the definition of near invariance, $ U_\theta f_2= F_2 $ for some $F_2 \in \mathcal{M} \text{ with } f_2= U_\theta^* F_2$ and $F_2 \perp \theta.$
		Furthermore, \begin{align*}
		& F_2= U_\theta ( F_1-a_1g )\\
		& = U_\theta ( I - g\otimes g )F_1 = U_\theta( I-g\otimes g) U_\theta (I- g\otimes g) (f ) = R_g^2(f)\\
		& =T ( F_1 ) - \langle F_1, g \rangle \mathbb{S}^*g = \mathbb{S}^* ( F_{1-} ) + S (F_{1_+}) - \langle F_{1-}, g \rangle \mathbb{S}^*g\\
		& = r_g( F_{1-} ) + zF_{1+} = r_g^2( f_- ) + z^2 f_+.
		\end{align*}
		Thus, $$ f= a_0g + a_1 U_\theta^* g + {U_\theta^*}^2 F_2 $$
		$$ \hspace{0.68in} = (a_0 + a_1 \bar{z})g + \mathbb{S}^2 r_g^2(f_-) + f_+, $$ and due to the orthogonality of each term in the summation, we have
		$$ ||f||^2= |a_0|^2 + |a_1|^2 + ||f_+||^2 + ||r_g^2(f_-)||^2.$$
		By continuing this iterative process, after $N$ steps, we arrive at
		\begin{align*}
		& f=g \sum_{k=0}^N a_k \bar{z}^k + {U_\theta^*}^{N+1} F_{N+1}\\
		& \hspace{0.1in} = g \sum_{k=0}^N a_k \bar{z}^k + {U_\theta^*}^{N+1} R_g^{N+1} ( f )\\
		& \hspace{0.1in} = g \sum_{k=0}^N a_k \bar{z}^k + f_+ + \mathbb{S}^{N+1} r_g^{N+1}(f_-), 
		\end{align*}
		and
		$$  ||f||^2 = \sum_{k=0}^N|a_k|^2 + ||f_+||^2 + || r_g^{N+1}(f_-)||^2. $$
		Therefore, we can directly conclude that $ \sum\limits_{k=0}^\infty|a_k|^2 < \infty,$ and hence $ \sum\limits_{k=0}^\infty a_k \bar{z}^k = \bar{h} $ for some element $h\in H^2$. 

        Now observe that the operator $\mathbb{S}$ is a $C_{.0}$ operator, that is, for every $x \in \overline{H_0^2}$, $||\mathbb{S^*}^N(x)|| \rightarrow 0$ for $N \rightarrow \infty$ and $dim(I-\mathbb{S}\mathbb{S}^*)=1$. Also, $\mathcal{M}_- \ominus (\vee\{g\})$ is a closed subspace of finite codimension. Therefore, by Lemma 3.3 of \cite{BT}, we have $\mathbb{S}P_{\mathcal{M}_- \ominus (\vee\{g\})}$ is a $C_{.0}$ operator, that is, $|| (P_{\mathcal{M}_- \ominus (\vee\{g\})} \mathbb{S}^*)^N(x) || \rightarrow 0 $ as $N \rightarrow \infty$ for every $x \in \overline{H_0^2}$. Thus, $$|| r_g^{N+1}(f_-) || = || \mathbb{S}^* \big((I - g \otimes g) \mathbb{S}^*\big)^N (I- g \otimes g)(f_-)|| \rightarrow 0 \text{ as } N \rightarrow \infty.$$ Therefore, letting $N \rightarrow \infty,$ we have $f= g\bar{h} + f_+ $ because $ ||\mathbb{S}^{N+1} r_g^{N+1}(f_-) || \rightarrow 0 $. 
        
        Hence, the nearly $U_\theta$-invariant subspace of $K_\theta^\perp$ has the structure $\mathcal{M}= g \mathcal{N}\oplus \mathcal{M}_+,$ where $\mathcal{M}_+ $ is a closed subspace of $\theta H^2$, and $ \mathcal{N}$ is a closed subspace of  $\overline{H^2}.$ To obtain the required structure, it remains to show that $\bar{z}\mathcal{ N }\oplus \mathcal{M}_+ $ is $ U_\theta$-invariant. Since, $\theta_0=0,$ then by Proposition \ref{3TH} it is equivalent of showing that $ \mathcal{M}_+ $ is an $S$-invariant subspace of $\theta H^2$ and $\mathcal{N}$ is an $\mathscr{S}^*$ (defined in Section 2) invariant subspace of $\overline{H^2}.$
		
		Note that, $f= f_- + f_+ = g\bar{h} + f_+ \implies f_-= g\bar{h}.$ Then, $r_g(f_-)= \mathbb{S}^*(g\bar{h})- a_0 \mathbb{S}^*(g),$ and $a_0 = \bar{h}_0.$
		On the other hand, \begin{align*}
		& \mathbb{S}^*(\bar{h}g)(z)= z\Bigl(\bar{h}g(z)- (\bar{h}g)_{-1}\bar{z}\Bigr)\\
		& = z\Bigl( \bar{h}(z)g(z) - \bar{h}_0g(z) + \bar{h}_0g(z) - \bar{h}_0g_{-1}\bar{z}  \Bigr)\\
		& = g(z) \mathscr{S}^*(\bar{h}) + \bar{h}_0 \mathbb{S}^* g.
		\end{align*}
		Therefore, \begin{align*}
		& r_g(f_-) = g(z) \mathscr{S}^*(\bar{h}) + \bar{h}_0 \mathbb{S}^* g - \bar{h}_0 \mathbb{S}^* g\\
		& = g(z) \mathscr{S}^*(\bar{h}).
		\end{align*}
		By Lemma \ref{3L3} we conclude $R_g{\mathcal{M}}\subset \mathcal{M},$ and hence  $r_g(f_-) \in \mathcal{M}_-=g\mathcal{N}$ and $zf_+ \in \mathcal{M}_+.$ Thus, $\mathcal{M}_+$ is an $S$ invariant subspace of $\theta H^2$ and it is of the form: $\{0\}, ~\gamma\theta H^2 \text{ or } ~\theta H^2$ for some inner function $\gamma.$ Whereas, $\frac{r_g(f_- )}{g} = \mathscr{S}^* (\bar{h}) \in \mathcal{N},$ proving that $\mathcal{N}$ is of the form: $\{0\}, ~\overline{K_\alpha}~\text{ or } ~\overline {H^2},$ for some inner function $\alpha$. Therefore, $\mathcal{M}_-$ is one of the form: $ \{0\}, ~g\overline{K_\alpha}~ \text{ or } ~{g_1}\overline{H_0^2}$ (where $g=\bar{z}{g_1}$, and $\overline{g_1}$ is inner with $\overline{g_1}(0) \neq 0$). This completes the proof. 
	\end{proof}
	
	Proceeding as above, we also obtain the structure of the nearly $U_\theta^*$-invariant subspaces of $ K_\theta^\perp$. It is not difficult to see that, if $\theta_0\neq 0,$ then nearly $U_\theta^*$-invariant subspaces are $U_\theta^*$ invariant subspaces of $K_\theta^\perp,$ and hence we get the structure of them using Theorem \ref{3Th}.
	
	\begin{corollary}(Case $\theta_0=0$)
		Let $\m \subseteq K_\theta^\perp$ be a closed subspace of the form $\m=\m_-\oplus \m_+$ which is nearly invariant under the action of the dual compressed backward shift $U_\theta^*$. Then $\m_-= \{0\}, \overline{H_0^2}, \text{ or }\overline{\alpha H_0^2}$ and $\m_+= \{0\}, \theta H^2, GK_\beta,$ or $ \gamma\theta H^2$ (in the last case, $G=\gamma$) for some inner functions $\alpha, \beta,$ and $ \gamma~ (\gamma_0 \neq 0)$ and some particular element $G$ $\in \theta H^2$.
	\end{corollary}

	In their paper, Câmara-Ross \cite{CR} showed that if $\theta_0 = 0$, then $U_\theta$ is unitarily equivalent to the operator $S \oplus S^*$ acting on the vector-valued Hardy space $H_{\mathbb{C}^2}^2$. Consequently, when a nearly $U_\theta$-invariant subspace of $K_\theta^\perp$ is transferred to a nearly $S \oplus S^*$-invariant subspace of $H_{\mathbb{C}^2}^2$, one obtains the so-called splitting subspaces, those of the form $\mathcal{N}_1 \oplus \mathcal{N}_2$, where $\mathcal{N}_1$ and $\mathcal{N}_2$ are closed subspaces of $H^2$ (see \cite{DT} for reference).

	\begin{corollary}\label{3mainconsequence}
		Let $\mathcal{M}$ be a non-trivial splitting subspace of $H_{\mathbb{C}^2}^2$ that is nearly invariant under the action of the operator $S\oplus S^*$. Then $\mathcal{M}$ has one of the following form: $ \alpha H^2 \oplus K_\beta,~ \{0\}\oplus K_\beta,~ \alpha H^2 \oplus \tilde{g}K_\beta, ~ \{0\}\oplus \tilde{g}K_\beta,~ \alpha H^2 \oplus \gamma H^2,~ \{0\} \oplus \gamma H^2 $ where $\alpha, \beta$, and $\gamma$ are inner functions with $\gamma_0 \neq 0,$ and some specific element $\tilde{g}\in H^2$.
	\end{corollary}
	This naturally leads to the following open problem.
	\vspace{0.1in}
	
	\textbf{Open Problem}: Taking  Definition \ref{3D} into account, figure out all the non-splitting nearly $S\oplus S^*$-invariant subspaces of $H_{\mathbb{C}^2}^2$.
	
	\subsection{ Near invariance of some other operators }
	In this subsection, we discuss near invariance corresponding to the operator $T^*_{{B_n}}= T_{\overline{B_n}}$ ($B_n$ is a finite Blaschke product) acting on the Hardy space $H^2$ using Definition \ref{3D}.
	
	Now according to Definition \ref{3D}, a closed subspace $\f$ of $H^2$ is a nearly $T^*_{{B_n}}$-invariant if $ T^*_{{B_n}} ( \f \cap B_n H^2 ) \subset \f,$ where $B_n$ is a finite Blaschke product.
	
	\begin{lemma}
		If a closed subspace $\f \subseteq H^2$ is a nearly $T^*_{{B_n}}$- invariant, then $\f$ can not be contained in $B_n H^2$.
	\end{lemma}
	
	\begin{proof}
		Suppose $\f$ is a nearly $T^*_{{B_n}}$-invariant subspace of $H^2$ such that $\f \subseteq B_nH^2.$ Then, by using the von Neumann-Wold decomposition \cite{ARO}, any element $f\in \f$ can be expressed as 
		$$ f=f_0+ B_n f_1 + B_n^2f_2+B_n^3 f_3 + ..., $$ 
		where the elements $f_m \in K_{B_n}$ for $m\in \mathbb{N} \cup \{0\}$. 
		
		Since $\f \subseteq B_nH^2,$  we have $f_0=0.$ Therefore, by using the definition of nearly $T^*_{{B_n}}$- invariant subspace, $ ({T^*_{{B_n}}})^k (f)= T_{\bn}^k (f) \in B_n H^2$ for every $k\in \mathbb{N}.$ This leads to the fact that $f_k=0$ for every $k \in \mathbb{N}\cup\{0\},$ and hence $f=0.$
	\end{proof}
	Now, we can apply the Hitt-Sarason algorithm to achieve the structure $\f.$ 
	\begin{theorem}
		Let $\f$ be a closed subspace of $H^2$ that is nearly invariant under the action of $T^*_{{B_n}}$. Then $\f$ has the structure $G\mathscr{N}$, where $ F (\in \mathscr{N})= \sum\limits_{k=0}^\infty c_k B_n^k \in H^2 $ ($c_k \in \mathbb{C},$ and $B_n$ is the finite Blaschke product, considered earlier) with $||F||^2= \sum\limits_{k=0}^\infty |c_k|^2$, and $G = P_{ (\f \ominus(\f \cap B_nH^2)) }(F)$. Moreover, $\mathscr{N}$ is an invariant subspace of $T^*_{{B_n}}$.
	\end{theorem}
	
	\begin{proof}
		We omit the details, as the proof is essentially identical to Hitt–Sarason’s argument, with $z$ replaced by $B_n$ (also see Theorem 3.4 of \cite{LP}).
	\end{proof}
	The above result is analogous to that obtained by Liang--Partington in \cite{LP}. 
	Recall that Liang--Partington introduced the notion of a nearly $T^{-1}$-invariant subspace 
	of a Hilbert space $\mathcal{H}$, where $T^{-1}$ denotes the left inverse of a bounded operator 
	$T$ acting on $\mathcal{H}$. This subsection aims to demonstrate that, by employing 
	Definition~\ref{3D} from the introduction, one can recover results similar to those of 
	Liang--Partington \cite{LP}. This is because if $T$ is an isometry, then $T^*$ 
	serves as the left inverse of $T$; hence, in this setting, the two notions, nearly 
	$T^{-1}$-invariant and nearly $T^*$-invariant (Definition~\ref{3D}), coincide.

	\section*{\textbf{{Acknowledgment}}}
	We thank the anonymous referee for the valuable comments and suggestions. A. Chattopadhyay is supported by the Core Research Grant (CRG), File No: CRG/2023/004826, by the Science and Engineering Research Board(SERB), Department of Science \& Technology (DST), Government of India. S. Jana gratefully acknowledges the support provided by IIT Guwahati, Government of India.
    \vspace{0.1in}

    {\bf Data Availability Statement:} This article is theoretical in nature, so no datasets were generated or analyzed. 
 
 \vspace{0.1in}
 
    {\bf Declaration}
     \vspace{0.1in}

    {\bf Conflicts of Interest:} The author declares that there are no conflicts of interest.

\vspace{0.1in}

	\noindent $^*$ [ A. Chattopadhyay ] Department of Mathematics, Indian Institute of Technology Guwahati, Guwahati, 781039, India.\\
	\textit{Email address:}~ arupchatt@iitg.ac.in, 2003arupchattopadhyay@gmail.com.
	\vspace{0.05in}
	
	\noindent $^{**}$ [ S. Jana ] Department of Mathematics, Indian Institute of Technology Guwahati, Guwahati, 781039, India.\\
	\textit{Email address:}~ supratimjana@iitg.ac.in, suprjan.math@gmail.com.
	
\end{document}